\documentclass[11pt,a4paper]{article}
\usepackage{epsf,epsfig,amsfonts,amsgen,amsmath,amstext,amsbsy,amsopn,amsthm,lineno}
\usepackage[dvips]{color}
\usepackage{graphicx}

\newtheorem{theorem}{Theorem}

\theoremstyle{definition}

\newtheorem{claim}{Claim}

\newtheorem{conjecture}{Conjecture}

\begin{document}

\title
{\bf {\Large  Subdivisions of vertex-disjoint cycles in bipartite
graphs }\thanks{Supported by NSFC (No. 11001214).}}
\date{}
\author{Shengning Qiao\thanks{Corresponding author. E-mail address: snqiao@xidian.edu.cn (S. Qiao).}  \\[2mm]
\small School of Mathematics and Statistics,\\
\small Xidian University,\\
\small  Xi'an 710071, Shaanxi, China \\ Bing
Chen \\[2mm]
\small Department of Applied Mathematics, School of Science,\\
\small Xi'an University of Technology,\\
\small  Xi'an 710048, Shaanxi, China \\} \maketitle

\begin{abstract}
Let $n\geq 6,k\geq 0$ be two integers. Let $H$ be a graph of order
$n$ with $k$ components, each of which is an even cycle of length at
least $6$ and $G$ be a  bipartite graph with bipartition $(X,Y)$
such that $|X|=|Y|\geq n/2$. In this paper, we show that if the
minimum degree of $G$ is at least $n/2-k+1$, then $G$ contains a
subdivision of $H$. This generalized an older result of Wang.
\medskip

\smallskip

\end{abstract}

\section{Introduction}

We use Bondy and Murty \cite{Bondy} for terminology and notation not
defined here and consider finite simple graphs only.

Let $G$ be a graph. A set of subgraphs of $G$ is said to be {\it
vertex-disjoint} if no two of them have a common vertex in $G$. Wang
\cite{Wang10} considered the minimum degree condition for existence
of vertex-disjoint large cycles in a bipartite graph, as follows.

\begin{theorem}[Wang \cite{Wang10}]
Let $G$ be a bipartite graph with bipartition $(X,Y)$ such that
$|X|=|Y|\geq sk$, where  $s\geq 3$ and  $k\geq 1$ are two integers.
If the minimum degree of $G$ is at least $(s-1)k+1$, then $G$
contains $k$ vertex-disjoint cycles of length at least $2s$.
\end{theorem}

A {\em subdivision } of a graph $H$ is a graph that can be obtained
from $H$ by a sequence of edge subdivisions. Let $H$ be a graph on
$n$ vertices and $k$ non-trivial components such that every
component contains at most one cycle. A {\em cyclic subdivision } of
$H$ is one such that only cyclic edges of $H$ are subdivided.

Recently, Babu and Diwan \cite{Babu} gave the following result.

\begin{theorem}[Babu and Diwan \cite{Babu}]
Let $H$ be a graph with $n$ vertices and $k$ non-trivial components
such that every component contains at most one cycle. Let $G$ be a
graph with at least $n$ vertices. If the minimum degree of $G$ is at
least $n-k$, then $G$ contains a cyclic subdivision of $H$.
\end{theorem}

Motivated by Theorem 2, this paper focuses on the analogous problem
for existence of subdivisions of vertex-disjoint cycles in bipartite
graphs. Our result is as follows.

\begin{theorem}
Let $H$ be a graph of order $n$ with $k$ components, each of which
is an even cycle of length at least $6$. Suppose that $G$ is a
bipartite graph with bipartition $(X,Y)$ such that $|X|=|Y|\geq
n/2$. If the minimum degree of $G$ is at least $n/2-k+1$, then $G$
contains a subdivision of $H$.
\end{theorem}

It is easy to see that Theorem 3 generalized Theorem 1. In fact, we
conjecture the following.

\begin{conjecture}
Let $H$ be a graph of order $n$ with $k$ components, each of which
is an even cycle. Suppose that $G$ is a bipartite graph with
bipartition $(X,Y)$ such that $|X|=|Y|\geq n/2$. If the minimum
degree of $G$ is at least $n/2-k+1$, then $G$ contains a subdivision
of $H$.
\end{conjecture}

If the Conjecture is true, then the minimum degree condition is
sharp. Let $k$ be an even. Let $H$ be an union of $k$ cycles such
that $k-1$ cycles of length 4 and a cycle of length 6. Let $G$ be a
bipartite graph with bipartition $(X_1\cup X_2\cup \{u\},Y_1\cup
Y_2\cup\{v\})$ such that the two induced subgraphs by $X_1\cup Y_1$
and $X_2\cup Y_2$ are isomorphic to $K_{k,k}$. Further, $u$ is
adjacent to every vertex in $Y_1\cup \{v\}$ and $v$ is adjacent to
every vertex in $X_2\cup \{u\}$. The vertices of $X_1$ are matched
with vertices of $Y_2$ by $k$ independent edges of $G$. Clearly, $G$
contains exactly $4k+2$ vertices and the minimum degree of $G$ is
$k+1<k+2$. It is not difficult to see that if $G$ contains $H$ then
the edge $uv$ will be contained by the cycle of length exactly 6.
Let $C$ be the cycle of length exactly 6 with contained the edge
$uv$. But, $G-V(C)$ does not contain a spanning subgraph with a
union of $k-1$ cycles of length 4.

We postpone the proof of Theorem 3 to the next section.

\section{Proof of Theorem 3}

We first introduce some further terminology and notations that will
be used later.

Let $G$ be a graph. The order of $G$ is denoted by $|G|$. For a
vertex $v$ and a subgraph $H$ of $G$, we use $N(v,H)$ and $d(v,H)$
to denote the set of vertices and the number of vertices in $H$
which are adjacent to $v$, respectively. Thus $d(v,G)$ is the degree
of $v$ in $G$. By $\delta(G)$ we denote the minimum degree of $G$.
For a subset $S$ of the vertices or the edges in $G$, we use $G[S]$
to denote the subgraph of $G$ induced by $S$. If $v\not\in V(H)$,
then we denote $G[V(H)\cup \{v\}]$ by $H+v$. If $v\in V(H)$, then we
denote $G[V(H)\setminus \{v\}]$ by $H-v$. By $G-H$ we denote
$G[V(G)\setminus V(H)]$. Let $x,y$ be two nonadjacent vertices in
$G$. We use $G+xy$ to denote the graph by adding the edge $xy$ in
$G$. Let $C$ be a cycle with a given orientation. We use $v^+$
(resp. $v^-$) to denote the successor (resp. the predecessor) of $v$
along $C$ according to this orientation.

{\bf Proof of Theorem 3.} By contradiction. Suppose that $G$ is a
graph satisfying the conditions of the theorem but containing no
subdivisions of $H$ such that the number of edges of $G$ is as large
as possible. This implies that that $G$ is not a complete bipartite
graph. Let $a\in X$ and $b\in Y$ be two nonadjacent vertices of $G$.
Denote the components of $H$ by $C_1,C_2,\ldots, C_k$. It follows
from the choice of $G$ that $G+ab$ contains a subdivision $H^*$ of
$H$. Let $C^*_i$ be the subdivision of $C_i$ in $H^*$ for $1\leq
i\leq k$. Without loss of generality, we assume that $ab\in
E(C^*_k)$. Then $G$ contains a subdivision of $H-C_k$ with
$$
\sum_{i=1}^{k-1}|C^*_i|\leq |G|-|C_k|.
$$
Let $G_1$ be the subgraph of $G$ induced by
$\bigcup_{i=1}^{k-1}V(C^*_i)$ and
$$
\beta(G_1)=\sum_{i=1}^{k-1}|E(G[V(C^*_i)])|.
$$
Set $G_2=G-G_1$ and let $P$ be a longest path in $G_2$. We give an
orientation to each $C^*_i$ for $1\leq i\leq k-1$.

We assume that
$C^*_1, C^*_2, \ldots, C^*_{k-1}$ are chosen in $G$ such that\\
(i) $|G_1|$ is as small as possible;\\
(ii) $P$ is as long as possible, subject to (i);\\
(iii) $\beta(G_1)$ is as large as possible, subject to (i) and (ii).

First, we prepare some claims.

\begin{claim}
Let $u$ be a vertex of $G_2$ and $v$ be a vertex of $C^*_i$ with
$1\leq i\leq k-1$. If $|C^*_i|>|C_i|$, then $d(u,C^*_i)<|C_i|/2$ and
$d(v,C^*_i)< |C_i|/2$. If $d(u,C^*_i)=|C_i|/2$ or
$d(v,C^*_i)=|C_i|/2$, then $|C^*_i|=|C_i|$.
\end{claim}

\begin{proof}
If $d(u,C^*_i)\geq |C_i|/2$, then it is not difficult to see that
$C^*_i+u$ contains a cycle of length less than $|C^*_i|$ but at
least $|C_i|$, which can be seen as a subdivision of $C_i$. This
contradicts the choice of $C^*_1, C^*_2, \ldots, C^*_{k-1}$ in (i).
If $d(v,C^*_i)\geq |C_i|/2$, then we can get a similar
contradiction. So we have $d(u,C^*_i)<|C_i|/2$ and $d(v,C^*_i)<
|C_i|/2$. The result $|C^*_i|=|C_i|$ if $d(u,C^*_i)=|C_i|/2$ or
$d(v,C^*_i)=|C_i|/2$ follows immediately.
\end{proof}

By Claim 1, we can immediately obtain that for any vertices $u\in
V(G_2),v\in V(C^*_i)$, we have $d(u,C^*_i)\leq |C_i|/2$ and
$d(v,C^*_i)\leq |C_i|/2$ for each $i$ with $1\leq i\leq k-1$.

\begin{claim}
Let $u'$ be a vertex in $V(G_2)\cap X$ and $u''$ be a vertex in
$V(G_2)\cap Y$. If $d(u',G_2)+d(u'',G_2)<|C_k|$, then there exists
an $C^*_i$ with $1\leq i\leq k-1$ such that
$d(u',C^*_i)+d(u'',C^*_i)\geq |C_i|-1$. Moreover, we have
$|C^*_i|=|C_i|$, and there exists a vertex $v\in N(u',C_i^*)$ such
that $C^*_i+u''-v$ contains $C_i$.
\end{claim}

\begin{proof}
If $d(u',C^*_j)+d(u'',C^*_j)\leq |C_j|-2$ for all $j$ with $1\leq
j\leq k-1$, then
\begin{eqnarray}
\nonumber & &d(u',G)+d(u'',G)\\
\nonumber &=& d(u',G_1)+d(u'',G_1)+d(u',G_2)+d(u'',G_2) \\
\nonumber &=& \sum\limits_{j=1}^{k-1}(d(u',C^*_j)+d(u'',C^*_j))+d(u',G_2)+d(u'',G_2) \\
\nonumber &<&\sum\limits_{j=1}^{k-1} (|C_j|-2)+|C_k| \\
\nonumber &=& n-2k+2,
\end{eqnarray}
contradicting the condition $\delta(G)\geq n/2-k+1$. Thus, there
exists an $C^*_i$ with $1\leq i\leq k-1$ such that
$d(u',C^*_i)+d(u'',C^*_i)\geq |C_i|-1$. This implies that
$d(u',C^*_i)=|C_i|/2$ or $d(u'',C^*_i)=|C_i|/2$. By Claim 1 we have
$|C^*_i|=|C_i|$. If $d(u'',C^*_i)=|C_i|/2$, then any $v\in
N(u',C^*_i)$ can be chosen as the required vertex. If
$d(u'',C^*_i)=|C_i|/2-1$, then $d(u',C^*_i)=|C_i|/2$. Since
$|C_i|\geq 6$, we can also find a vertex $v\in N(u',C^*_i)$ such
that $C^*_i+u''-v$ contains $C_i$.
\end{proof}

Now let $P=u_1u_2\cdots u_s$ and assume that $u_1\in X$.

\begin{claim}
$P$ is a Hamilton path of $G_2$.
\end{claim}

\begin{proof}
Suppose that $P$ is not a Hamilton path of $G_2$. Then from the
choice of $P$ and the fact that $G_2$ does not contain a subdivision
of $C_k$, we have
$$
d(u_1,G_2)=d(u_1,P)\leq \frac{|C_k|}{2}-1
$$
and
$$
d(u_s,G_2)=d(u_s,P)\leq \frac{|C_k|}{2}-1.
$$
Let $G'_2=G_2-V(P)$ and $M=\{a_1b_1, a_2b_2,\ldots, a_mb_m\}$ be a
maximum matching in $G'_2$ with $\{a_1, a_2, \ldots, a_m\}\subseteq
X$. We distinguish two cases as follows.

\medskip
\noindent{\bf Case 1.} $M$ is not a perfect matching of $G'_2$.
\medskip

First note that $G_2$ has an even number of vertices. If $u_s\in Y$,
then $P$ has also an even number of vertices. Thus, $|G'_2|$ is even
and $|V(G'_2)\cap X|=|V(G'_2)\cap Y|$. If $u_s\in X$, then
$|V(G'_2)\cap Y|=|V(G'_2)\cap X|+1$. Since $M$ is not a perfect
matching of $G'_2$, we have $(V(G'_2)\cap Y)\setminus
\{b_1,b_2,\ldots,b_m\}\neq \emptyset$.

Let $y_0$ be a vertex in $(V(G'_2)\cap Y)\setminus
\{b_1,b_2,\ldots,b_m\}$. Choose $P_1$ as a longest $M$-alternating
path in $G'_2$ starting from $y_0$. Then $P_1$ must end at a vertex
$y_1$ with $y_1\in \{y, b_1, \ldots, b_m\}$. Let $M'=M-E(P_1)$ and
$G''_2=G'_2-V(P_1)+y$. Choose $P_2$ as a longest $M'$-alternating
path in $G''_2$ starting from $y_0$. Then $P_2$ must end at a vertex
$y_2\in \{y,b_1, \ldots, b_m\}$. Therefore, $Q=P_1\cup P_2$ is a
path from $y_1$ to $y_2$ in $G'_2$ such that $d(y_1, G'_2)=d(y_1,Q)$
and $d(y_2, G'_2)=d(y_2,Q)$. Since $G_2$ does not contain a
subdivision of $C_k$, we have
$$
d(y_1,Q)\leq \frac{|C_k|}{2}-1
$$
and
$$
d(y_2,Q)\leq \frac{|C_k|}{2}-1.
$$

If $d(y_1,G_2)\leq |C_k|/2$, then by $d(u_1,G_2)+d(y_1,G_2)<|C_k|$
and Claim 2 we know that $d(u_1,C^*_i)+d(y_1,C^*_i)\geq |C_i|-1$ for
some $i$ with $1\leq i\leq k-1$, and there exists a vertex $v\in
N(u_1,C^*_i)$ such that $C^*_i+y_1-v$ contains $C^*_i$. This
contradicts the choice of $C^*_1, C^*_2, \ldots, C^*_{k-1}$ in (ii).
Hence, we have $d(y_1,G_2)\geq |C_k|/2+1$. Similarly we have
$d(y_2,G_2)\geq |C_k|/2+1$. This implies that
$$
d(y_1,P)=d(y_1,G_2)-d(y_1,G'_2)\geq |C_k|/2+1-d(y_1,Q)\geq 2
$$
and
$$
d(y_2,P)=d(y_2,G_2)-d(y_2,G'_2)\geq |C_k|/2+1-d(y_2,Q)\geq 2.
$$
It is not difficult to find a segment of $P$ with at least
$$
2min\{\frac{|C_k|}{2}+1-d(y_1,Q), \frac{|C_k|}{2}+1-d(y_2,Q)\}-1
$$
vertices such that the two end-vertices of this segment are adjacent
to $y_1$ and $y_2$, respectively. Since $Q$ has at least
$$
2max\{d(y_1,Q), d(y_2,Q)\}
$$
vertices, we can obtain a cycle with at least $|C_k|$ vertices in
$G[V(P\cup Q)]$, which can be seen as a subdivision of $C_k$. This
contradicts our assumption that $G$ contains no subdivisions of $H$.

\medskip
\noindent{\bf Case 2.} $M$ is a perfect matching of $G'_2$.
\medskip

Choose $P_3$ as a longest $M$-alternating path in $G'_2$ such that
the first edge of $P_3$ is in $M$. Denote the initial vertex and the
final vertex of $P_3$ by $x$ and $y$, respectively. Then we have
$x\in\{a_1,a_2,\ldots ,a_m\}$ and $y\in \{b_1,b_2,\ldots, b_m\}$.
Since $G_2$ does not contain a subdivision of $C_k$, we know that
$$
d(x,P_3)\leq \frac{|C_k|}{2}-1
$$
and
$$
d(y,P_3)\leq \frac{|C_k|}{2}-1.
$$
Clearly, $u_s\in Y$. If $d(x,G_2)\leq |C_k|/2$, then it follows from
$d(u_s,G_2)+d(x,G_2)<|C_k|$ and Claim 2 that
$d(u_s,C^*_i)+d(x,C^*_i)\geq |C_i|-1$ for some $i$ with $1\leq i\leq
k-1$, and there exists a vertex $v\in N(u_s,C^*_i)$ such that
$C^*_i+x-v$ contains $C^*_i$. This contradicts the choice of $C^*_1,
C^*_2, \ldots, C^*_{k-1}$ in (ii). Hence, we have $d(x,G_2)\geq
|C_k|/2+1$. Similarly, we can show that $d(y,G_2)\geq |C_k|/2+1$.
This means that
$$
d(x,P)\geq |C_k|/2+1-d(x,P_3)\geq 2
$$
and
$$
d(y,P)\geq |C_k|/2+1-d(y,P_3)\geq 2.
$$
The rest of the proof of this case is just as same as that of Case
1.
\end{proof}

Since $G_2$ does not contain a subdivision of $C_k$ and $P$ is a
Hamilton path of $G_2$, we can immediately get
\begin{eqnarray}
 d(u_1,G_2)\leq |C_k|/2-1,& &d(u_s,G_2)\leq
|C_k|/2-1
\end{eqnarray}

\begin{eqnarray}
 d(u_2,G_2)\leq |C_k|/2,& &d(u_{s-1},G_2)\leq
|C_k|/2
\end{eqnarray}

\noindent Note $u_s\in Y$. Then, by (1) and Claim 2, we know that
there exists a component $C^*_p$ for some $p$ with $1\leq p\leq k-1$
such that
$$
d(u_1, C^*_p)+ d(u_s, C^*_p)\geq |C_p|-1
$$
and
$$
|C^*_p|=|C_p|.
$$
Without loss of generality, we may assume $d(u_1, C^*_p)=|C_p|/2$
and $d(u_s, C^*_p)\geq |C_p|/2-1$. Let $x^*$ be a vertex of
$V(C^*_p)\cap X$ such that $u_s$ is adjacent to all vertices in
$(V(C^*_p)\cap X)\setminus\{x^*\}$.  Let $y^*$ be a vertex of
$V(C^*_p)\cap Y$ such that $x^*$ and $y^*$ are not adjacent on
$C^*_p$.

We have the following Claim 4 and Claim 5.

\begin{claim}
For all $z\in \{u_1, u_{s-1}, u_s, x^*,y^* \}$ we have
$$
d(z, G_2)+d(z, C^*_p)\leq \frac{|C_k|}{2}+\frac{|C_p|}{2}-1,
$$
and for $u_2$ we have
$$
d(u_2, G_2)+d(u_2, C^*_p)< \frac{|C_k|}{2}+\frac{|C_p|}{2}-1.
$$
\end{claim}

\begin{proof}
The case $z=u_1$ or $z=u_s$ follows from (1) and the fact
$|C^*_p|=|C_p|$ immediately.

Suppose $z=x^*$. Then we have $u_1x^*\not \in E(G)$. If $x^*$ is
adjacent to at least $|C_k|/2$ vertices in $P$, then $P+x^*-u_1$
contains a subdivision of $C_k$ and $C^*_p+u_1-x^*$ contains $C_p$.
This contradicts our assumption that $G$ contains no subdivisions of
$H$. Thus, we have
$$
d(x^*, G_2)+d(x^*, C^*_p)\leq \frac{|C_k|}{2}+\frac{|C_p|}{2}-1.
$$
Similarly, we can prove
$$
d(y^*, G_2)+d(y^*, C^*_p)\leq \frac{|C_k|}{2}+\frac{|C_p|}{2}-1.
$$

Suppose $z=u_{s-1}$. If $u_{s-1}$ is adjacent to a vertex $v$ in
$(V(C^*_p)\cap Y)\setminus\{x^{*+},x^{*-}\}$, then from $u_{s-1}\in
X$ and $|C^*_p|\geq 6$ we know that $G_2+v-u_s$ contains a
subdivision of $C_k$ and $C^*_p+u_s-v$ contains $C_p$. This
contradicts our assumption that $G$ contains no subdivisions of $H$.
Together with (2), we can obtain
$$
d(u_{s-1}, G_2)+d(u_{s-1}, C^*_p)\leq \frac{|C_k|}{2}+2\leq
\frac{|C_k|}{2}+\frac{|C_p|}{2}-1.
$$

Note that $u_2\in Y$. If $u_2$ is adjacent to a vertex $v$ in
$(V(C^*_p)\cap X)\setminus\{x^*\}$, then $G_2+v-u_1$ contains a
subdivision of $C_k$ and $C^*_p+u_1-v$ contains $C_p$, a
contradiction. Thus, together with (2) and $|C_p|\geq 6$, we have
$$
d(u_2, G_2)+d(u_2, C^*_p)\leq \frac{|C_k|}{2}+1<
\frac{|C_k|}{2}+\frac{|C_p|}{2}-1.
$$
\end{proof}

\begin{claim}
Let $S=\{u_1,u_2,u_{s-1},u_s,x^*,y^*\}$. Then there exists a $C^*_q$
with $q\neq p$ and $1\leq q\leq k-1$ such that
$$
\sum\limits_{z\in S} d(z, C^*_q)\geq 3|C_q|-5
$$
and $|C^*_q|=|C_q|$.
\end{claim}

\begin{proof}
By Claim 4, we can get that
$$
\sum\limits_{z\in S}d(z,C^*_p)+\sum\limits_{z\in
S}d(z,G_2)<3|C_k|+3|C_p|-6.
$$
If $\sum_{z\in S}d(z,C^*_i)\leq 3|C_i|-6$ for all $i$ with $1\leq
i\leq k-1$ and $i\neq p$, then we have
\begin{eqnarray}
\nonumber & &\sum\limits_{z\in S}d(z,G)\\ \nonumber
&=&\sum\limits_{z\in S}d(z,G_1-C^*_p)+\sum\limits_{z\in
S}d(z,C^*_p)+\sum\limits_{z\in
S}d(z,G_2)\\
\nonumber &=& \sum\limits_{i=1,i\neq p}^{k-1}\sum\limits_{z\in S}d(z,C^*_i)+\sum\limits_{z\in S}d(z,C^*_p)+\sum\limits_{z\in S}d(z,G_2) \\
\nonumber &<& 3(n-|C_k|-|C_p|)-6k+12+3|C_k|+3|C_p|-6 \\
\nonumber &=& 3n-6k+6.
\end{eqnarray}
This contradicts the condition $\delta(G)\geq n/2-k+1$. Hence, there
exists a $C^*_q$ for some $q$ with $1\leq q\leq k-1$ and $q\neq p$
such that
$$
\sum\limits_{z\in S} d(z, C^*_q)\geq 3|C_q|-5.
$$
Then, at least one of the following holds.
$$
d(u_1, C^*_q)+d(u_s, C^*_q)\geq |C_q|-1,
$$
$$d(u_2, C^*_q)+d(u_{s-1},
C^*_q)\geq |C_q|-1
$$
and
$$
d(x^*, C^*_q)+d(y^*, C^*_q)\geq |C_q|-1.
$$
If $d(u_1, C^*_q)+d(u_s, C^*_q)\geq |C_q|-1$ or $d(u_2,
C^*_q)+d(u_{s-1},C^*_q)\geq |C_q|-1$, then by Claim 2 we have
$|C^*_q|=|C_q|$. If $ d(x^*, C^*_q)+d(y^*, C^*_q)\geq |C_q|-1$, then
$d(x^*, C^*_q)=|C_q|/2$ or $d(y^*, C^*_q)=|C_q|/2$. Since
$C^*_p+u_1-x^*$ or $C^*_p+u_s-y^*$ contains $C_p$, we can prove
$|C^*_q|=|C_q|$ as in Claim 1.
\end{proof}

In order to complete the proof of Theorem 3, we will show that
$G[V(G_2)\cup V(C^*_p)\cup V(C^*_q)$ contains a subdivision of
$C_k\cup C_p\cup C_q$, which contradicts our assumption that $G$
contains no a subdivision of $H$. Note that $d(u_1, C^*_q)+d(u_s,
C^*_q)\leq |C_q|$, $d(u_2, C^*_q)+d(u_{s-1}, C^*_q)\leq |C_q|$ and $
d(x^*, C^*_q)+d(y^*, C^*_q)\leq |C_q|$.

We consider the two cases as follows.

\medskip
\noindent{\bf Case 1.} $d(u_1, C^*_q)+d(u_s, C^*_q)\geq |C_q|-1$.
\medskip

Suppose that $d(u_1, C^*_q)+d(u_s, C^*_q)=|C_q|$. Then by Claim 5 we
have $d(u_2, C^*_q)+d(u_{s-1}, C^*_q)\geq |C_q|-5\geq 1$. This
implies that either $u_1$ and $u_{s-1}$ have at least one common
neighbor in $C^*_q$ or $u_2$ and $u_s$ have at least one common
neighbor in $C^*_q$. By the symmetry, we assume that $v\in V(C^*_q)$
with $u_1v,u_{s-1}v\in E(G)$. It is easy to see that $G_2+v-u_s$
contains a subdivision of $C_k$ and $C^*_q+u_s-v$ contains a
subdivision of $C_q$.

Suppose that $d(u_1, C^*_q)+d(u_s, C^*_q)=|C_q|-1$. By the symmetry,
we assume that $d(u_1, C^*_q)=|C_q|/2$ and $d(u_s,
C^*_q)=|C_q|/2-1$. Let $v$ be the vertex in $V(C^*_q)\cap X$ with
$u_sv\not\in E(G)$. If $|C_q|\geq 8$, then by Claim 5 we have
$d(u_2, C^*_q)+d(u_{s-1}, C^*_q)\geq |C_q|-4\geq 4$. So, we can
deduce that either $u_2$ has at least one neighbor in
$(V(C^*_q)\setminus \{v\})\cap X$ or $u_{s-1}$ has at least one
neighbor in $(V(C^*_q)\setminus \{v^+,v^-\})\cap Y$. If $u_2$ has at
least one neighbor in $(V(C^*_q)\setminus \{v\})\cap X$, say $v_x$,
then $G_2+v_x-u_1$ contains a subdivision of $C_k$ and
$C^*_q+u_1-v_x$ contains a subdivision of $C_q$. If $u_{s-1}$ has at
least one neighbor in $(V(C^*_q)\setminus \{v^+,v^-\})\cap Y$, say
$v_y$, then $G_2+v_y-u_s$ contains a subdivision of $C_k$ and
$C^*_q+u_s-v_y$ contains a subdivision of $C_q$. Thus, we have
$|C_q|=6$. If $u_2$ or $u_{s-1}$ has a neighbor in
$V(C^*_q)\setminus \{v^-,v,v^+\}$, then the proof is as same as that
of the case $|C_q|\geq 8$. We may assume that $N(u_2,C^*_q)\cup
N(u_{s-1},C^*_q)\subseteq \{v^-,v,v^+\}$. Since $d(u_2,
C^*_q)+d(u_{s-1}, C^*_q)\geq |C_q|-4=2$, we have $u_{s-1}v^-\in
E(G)$ or $u_{s-1}v^+\in E(G)$. Without loss of generality, we assume
that $u_{s-1}v^+\in E(G)$. If $u_2v\not\in E(G)$, then $d(u_2,
C^*_q)=0$ and $d(u_{s-1}, C^*_q)= |C_q|-4$. By Claim 5 we have
$d(x^*, C^*_q)+d(y^*, C^*_q)=|C_q|$. It follows from the choice of
$y^*$ in $C^*_p$ that $G_2+v^+-u_s$ contains a subdivision of $C_k$,
$C^*_p+u_s-y^*$ contains a subdivision of $C_p$ and $C^*_q+y^*-v^+$
contains a subdivision of $C_q$. If $u_2v\in E(G)$, then we can
deduce that $v$ is adjacent to every vertex in $V(C^*_q)\cap Y$. If
not, then $P+v-u_1$ is a path of order $|P|$ and $C^*_q+u_1-v$
contains a subdivision of $C_q$ such that $|C^*_q+u_1-v|=|C^*_q|$
and $|E(G[V(C^*_q+u_1-v)])|>|E(G[V(C^*_q)])|$, which contradicts the
choice of $C^*_1, C^*_2, \ldots, C^*_{k-1}$ in (iii). Then it is not
difficult to see that $G_2+v^+-u_s$ contains a subdivision of $C_k$
and $C^*_q+u_s-v^+$ contains a subdivision of $C_q$.

\medskip
\noindent{\bf Case 2.} $d(u_1, C^*_q)+d(u_s, C^*_q)\leq |C_q|-2$.
\medskip

We will consider three subcases as follows.

\medskip
\noindent{\bf Case 2.1.} $d(u_1, C^*_q)+d(u_s, C^*_q)\leq |C_q|-4$.
\medskip

By Claim 5 we have
$$
\sum\limits_{z\in \{u_2,u_{s-1},x^*,y^*\}} d(z, C^*_q)\geq 2|C_q|-1.
$$
This implies that either $d(u_2, C^*_q)+d(u_{s-1}, C^*_q)=|C_q|$ and
$d(x^*, C^*_q)+d(y^*, C^*_q)\geq|C_q|-1$ or $d(u_2,
C^*_q)+d(u_{s-1}, C^*_q)\geq|C_q|-1$ and $d(x^*, C^*_q)+d(y^*,
C^*_q)=|C_q|$. If $d(u_2, C^*_q)+d(u_{s-1}, C^*_q)=|C_q|$ and
$d(x^*, C^*_q)+d(y^*, C^*_q)\geq|C_q|-1$, then there exists an edge
$v_xv_y\in E(C^*_q)$ such that $x^*$ is adjacent to every vertex in
$(V(C^*_q)\setminus \{v_y\})\cap Y$ and $y^*$ is adjacent to every
vertex in $(V(C^*_q)\setminus \{v_x\})\cap X$. Thus we can get that
$G_2+v_x+v_y-u_1-u_s$ contains a subdivision of $C_k$,
$C^*_p+u_1+u_s-x^*-y^*$ contains a subdivision of $C_p$ and
$C^*_q+x^*+y^*-v_x-v_y$ contains a subdivision of $C_q$. The case
$d(u_2, C^*_q)+d(u_{s-1}, C^*_q)\geq|C_q|-1$ and $d(x^*,
C^*_q)+d(y^*, C^*_q)=|C_q|$ can be proved similarly.

\medskip
\noindent{\bf Case 2.2.} $d(u_1, C^*_q)+d(u_s, C^*_q)=|C_q|-3$.
\medskip

By Claim 5, we have $d(x^*, C^*_q)+d(y^*, C^*_q)+d(u_2,
C^*_q)+d(u_{s-1}, C^*_q)\geq 2|C_q|-2$. This implies that $d(x^*,
C^*_q)+d(y^*, C^*_q)=|C_q|$, $d(u_2, C^*_q)+d(u_{s-1},
C^*_q)=|C_q|$, or $d(x^*, C^*_q)+d(y^*, C^*_q)=|C_q|-1$ and $d(u_2,
C^*_q)+d(u_{s-1}, C^*_q)=|C_q|-1$.  We consider three subcases as
follows.

\medskip
\noindent{\bf Case 2.2.1.} $d(x^*, C^*_q)+d(y^*, C^*_q)=|C_q|$.
\medskip

Then by Claim 5 we have $d(u_2, C^*_q)+d(u_{s-1}, C^*_q)\geq
|C_q|-2>|C_q|/2$. It is not difficult to see that there exists an
edge $v_xv_y\in E(C^*_q)$ with $u_2v_x,u_{s-1}v_y\in E(G)$ such that
$G_2+v_x+v_y-u_1-u_s$ contains a subdivision of $C_k$,
$C^*_p+u_1+u_s-x^*-y^*$ contains a subdivision of $C_p$ and
$C^*_q+x^*+y^*-v_x-v_y$ contains a subdivision of $C_q$.

\medskip
\noindent{\bf Case 2.2.2.} $d(u_2, C^*_q)+d(u_{s-1}, C^*_q)=|C_q|$.
\medskip

Then by Claim 5 we have $d(x^*, C^*_q)+d(y^*, C^*_q)\geq |C_q|-2$.

First, we assume that $|C_q|\geq 8$. Then we can see that $d(u_1,
C^*_q)\geq 1$. If $d(y^*, C^*_q)\geq|C_q|/2-1$, then we can deduce
that there exists an edge $v_xv_y\in E(C^*_q)$ with
$u_2v_x,u_{s-1}v_y\in E(G)$ such that $G_2+v_x+v_y-u_1-u_s$ contains
a subdivision of $C_k$, $C^*_p+u_s-y^*$ contains a subdivision of
$C_p$ and $C^*_q+u_1+y^*-v_x-v_y$ contains a subdivision of $C_q$.
If $d(y^*, C^*_q)=|C_q|/2-2$, then $d(x^*, C^*_q)=|C_q|/2$ and we
can also deduce that there exists an edge $v_xv_y\in E(C^*_q)$ with
$u_2v_x,u_{s-1}v_y\in E(G)$ such that $G_2+v_x+v_y-u_1-u_s$ contains
a subdivision of $C_k$, $C^*_p+u_1+u_s-x^*-y^*$ contains a
subdivision of $C_p$ and $C^*_q+x^*+y^*-v_x-v_y$ contains a
subdivision of $C_q$.

Next, we assume that $|C_q|=6$. If $d(u_1,C^*_q)=3$, then by
$d(y^*,C^*_q)\geq 1$ and $u_1y^*\in E(G)$ we can deduce that there
exists an edge $v_xv_y\in E(C^*_q)$ such that $u_2v_x,u_{s-1}v_y\in
E(G)$ and $G_2+v_x+v_y-u_1-u_2$ contains a subdivision of $C_k$,
$C^*_p+u_s-y^*$ contains a subdivision of $C_p$ and
$C^*_q+u_1+y^*-v_x-v_y$ contains a subdivision of $C_q$. So we
assume that $d(u_1,C^*_q)\leq 2$. Then $d(u_s,C^*_q)\geq 1$. Let
$v'_x$ be a vertex in $C^*_q$ with $u_sv'_x\in E(G)$. We claim that
if $d(x^*,C^*_q)=3$, then $x^*y^*\in E(G)$. If $x^*y^*\not\in E(G)$,
then we can  get that $P+v'_x-u_1$ is a path of order $|P|$,
$C^*_q+x^*-v'_x$ contains a subdivision of $C_q$ with
$|C^*_q+x^*-v'_x|=|C_q|$ and $|E(G[V(C^*_q+x^*-v'_x)])|\geq
|E(G[V(C^*_q)])|$, and $C^*_p+u_1-x^*$ contains a subdivision of
$C_p$ with $|C^*_p+u_1-x^*|=|C_p|$ and
$|E(G[V(C^*_p+u_1-x^*)])|>|E(G[V(C^*_p)])|$, which contradicts the
choice of $C^*_1, C^*_2, \ldots, C^*_{k-1}$ in (iii). Thus, if
$d(x^*,C^*_q)=3$, then by $d(y^*,C^*_q)\geq 1$, it is not difficult
to see that there exists an edge $v_xv_y\in E(C^*_q)$ such that
$u_2v_x,u_{s-1}v_y\in E(G)$ and $G_2+v_x+v_y-u_1-u_2$ contains a
subdivision of $C_k$, $C^*_p+u_1+u_s-x^*-y^*$ contains a subdivision
of $C_p$ and $C^*_q+x^*+y^*-v_x-v_y$ contains a subdivision of
$C_q$. If $d(y^*,C^*_q)=3$ and $d(u_1,C^*_q)\geq 1$, then we can
also obtain that there exists an edge $v_xv_y\in E(C^*_q)$ such that
$u_2v_x,u_{s-1}v_y\in E(G)$ and $G_2+v_x+v_y-u_1-u_2$ contains a
subdivision of $C_k$, $C^*_p+u_s-y^*$ contains a subdivision of
$C_p$ and $C^*_q+u_1+y^*-v_x-v_y$ contains a subdivision of $C_q$.
If $d(y^*,C^*_q)=3$ and $d(u_s,C^*_q)=3$, then we claim that
$x^*y^*\in E(G)$. If $x^*y^*\not\in E(G)$, then we can  get that
$P+v'_y+x^*-u_1-u_s$ is a path of order $|P|$, where $v'_y\in
N(x^*,C^*_q)$,  $C^*_q+u_s-v'_y$ contains a subdivision of $C_q$
with $|C^*_q+u_s-v'_y|=|C_q|$ and $|E(G[V(C^*_q+u_s-v'_y)])|\geq
|E(G[V(C^*_q)])|$, and $C^*_p+u_1-x^*$ contains a subdivision of
$C_p$ with $|C^*_p+u_1-x^*|=|C_p|$ and
$|E(G[V(C^*_p+u_1-x^*)])|>|E(G[V(C^*_p)])|$, which contradicts the
choice of $C^*_1, C^*_2, \ldots, C^*_{k-1}$ in (iii). Note that
$d(x^*,C^*_q)\geq 1$. So, there exists an edge $v_xv_y\in E(C^*_q)$
such that $u_2v_x,u_{s-1}v_y\in E(G)$ and $G_2+v_x+v_y-u_1-u_2$
contains a subdivision of $C_k$, $C^*_p+u_1+u_s-x^*-y^*$ contains a
subdivision of $C_p$ and $C^*_q+x^*+y^*-v_x-v_y$ contains a
subdivision of $C_q$. We may assume that
$d(x^*,C^*_q)=d(y^*,C^*_q)=2$. Let $C^*_q=v_1v_2v_3v_4v_5v_6v_1$
with $v_1\in X$. Without loss of generality, we suppose that
$v_2x^*, v_4x^*\in E(G)$. If $v_1y^*,v_3y^*\in E(G)$ or
$v_3y^*,v_5y^*\in E(G)$, by the symmetry, say $v_1y^*,v_3y^*\in
E(G)$, then $G_2+v_5+v_6-u_1-u_2$ contains a subdivision of $C_k$,
$C^*_p+u_1+u_s-x^*-y^*$ contains a subdivision of $C_p$ and
$C^*_q+x^*+y^*-v_5-v_6$ contains a subdivision of $C_q$. Thus let
$v_1y^*,v_5y^*\in E(G)$. If $u_1v_2$ or $u_1v_4\in E(G)$, by the
symmetry, say $u_1u_2\in E(G)$, then $G_2+v_3+v_4-u_1-u_2$ contains
a subdivision of $C_k$, $C^*_p+u_s-y^*$ contains a subdivision of
$C_p$ and $C^*_q+u_1+y^*-v_3-v_4$ contains a subdivision of $C_q$.
This implies that $d(u_1,C^*_q)\leq 1$ and $d(u_s,C^*_q)\geq 2$. If
$u_1v_6,u_sv_1,u_sv_5\in E(G)$, then we can see that $G_2+v_6-u_s$
contains a subdivision of $C_k$ and $C^*_q+u_s-v_6$ contains a
subdivision of $C_q$. Thus we assume that $u_sv_1,u_sv_3\in E(G)$ or
$u_sv_3,u_sv_5\in E(G)$. By the symmetry, we suppose that
$u_sv_1,u_sv_3\in E(G)$. Then it is not difficult to see that
$G_2+v_5+v_6-u_1-u_2$ contains a subdivision of $C_k$,
$C^*_p+u_1-x^*$ contains a subdivision of $C_p$ and
$C^*_q+u_s+x^*-v_5-v_6$ contains a subdivision of $C_q$.

\medskip
\noindent{\bf Case 2.2.3.} $d(x^*, C^*_q)+d(y^*, C^*_q)=|C_q|-1$ and
$d(u_2, C^*_q)+d(u_{s-1}, C^*_q)=|C_q|-1$.
\medskip

If $|C_q|\geq 8$, or $|C_q|=6$ and $N(x^*, C^*_q)\cup N(y^*,
C^*_q)\neq N(u_2, C^*_q)\cup N(u_{s-1}, C^*_q)$, then it is easy to
deduce that there exists an edge $v_xv_y\in E(C^*_q)$ with
$u_2v_x,u_{s-1}v_y\in E(G)$ such that $G_2+v_x+v_y-u_1-u_s$ contains
a subdivision of $C_k$, $C^*_p+u_1+u_s-x^*-y^*$ contains a
subdivision of $C_p$ and $C^*_q+x^*+y^*-v_x-v_y$ contains a
subdivision of $C_q$. Now suppose that $|C_q|=6$ and $N(x^*,
C^*_q)\cup N(y^*, C^*_q)=N(u_2, C^*_q)\cup N(u_{s-1}, C^*_q)$. Let
$C^*_q=v_1v_2v_3v_4v_5v_6v_1$ with $v_1\in X$. Without loss of
generality, we assume that $u_2v_1,v_1y^*\not\in E(G)$. Then we have
$u_2v_3,u_2v_5,v_3y^*,v_5y^*\in E(G)$. If $d(u_1,C^*_q)=3$, then we
can obtain that $G_2+v_3+v_4-u_1-u_s$ contains a subdivision of
$C_k$, $C^*_p+u_s-y^*$ contains a subdivision of $C_p$ and
$C^*_q+u_1+y^*-v_3-v_4$ contains a subdivision of $C_q$. If
$d(u_1,C^*_q)\leq 2$, then $d(u_s,C^*_q)\geq 1$, say $v_x\in
V(C^*_q)$ with $u_sv_x\in E(G)$. We claim that $x^*y^*\in E(G)$. If
not, then by $d(x^*,C^*_q)=3$ we can deduce that $P+v_x-u_1$ is a
path of order $|P|$, $C^*_q+x^*-v_x$ contains a subdivision of $C_q$
with $|C^*_q+x^*-v_x|=|C_q|$ and $|E(G[V(C^*_q+x^*-v_x)])|\geq
|E(G[V(C^*_q)])|$, and $C^*_p+u_1-x^*$ contains a subdivision of
$C_p$ with $|C^*_p+u_1-x^*|=|C_p|$ and
$|E(G[V(C^*_p+u_1-x^*)])|>|E(G[V(C^*_p)])|$, which contradicts the
choice of $C^*_1, C^*_2, \ldots, C^*_{k-1}$ in (iii). Thus we can
see that $G_2+v_3+v_4-u_1-u_s$ contains a subdivision of $C_k$,
$C^*_p+u_1+u_s-x^*-y^*$ contains a subdivision of $C_p$ and
$C^*_q+x^*+y^*-v_3-v_4$ contains a subdivision of $C_q$.

\medskip
\noindent{\bf Case 2.3.} $d(u_1, C^*_q)+d(u_s, C^*_q)=|C_q|-2$.
\medskip

By the symmetry, we let $d(u_1, C^*_q)\geq |C^*_q|/2-1$. Note that
$d(x^*, C^*_q)+d(y^*, C^*_q)\leq |C_q|$. It follows from Claim 5
that $|C_q|-3\leq d(u_2, C^*_q)+d(u_{s-1}, C^*_q)\leq|C_q|$.

First, we assume that $|C_q|\geq 8$. If $d(u_2, C^*_q)+d(u_{s-1},
C^*_q)\geq |C_q|-1$ and $d(u_1, C^*_q)=|C^*_q|/2$, then $d(u_s,
C^*_q)=|C^*_q|/2-2\geq 2$ and there exists at least one vertex $v_x$
in $V(C^*_q)\cap X$ such that $u_2v_x,u_sv_x\in E(G)$. So we can
obtain that $G_2+v_x-u_1$ contains a subdivision of $C_k$ and
$C^*_q+u_1-v_x$ contains a subdivision of $C_q$. If $d(u_2,
C^*_q)+d(u_{s-1}, C^*_q)\geq |C_q|-1$ and $d(u_1,
C^*_q)=|C^*_q|/2-1\geq3$, then $d(u_s, C^*_q)=|C^*_q|/2-1\geq 3$. By
the symmetry we let $d(u_2, C^*_q)=|C_q|/2$. Then there exists at
least one vertex $v_x$ in $V(C^*_q)\cap X$ such that $u_2v_x,u_sv_x,
u_1v^-_x,u_1v^+_x\in E(G)$. It is easy to see that $G_2+v_x-u_1$
contains a subdivision of $C_k$ and $C^*_q+u_1-v_x$ contains a
subdivision of $C_q$. If $d(u_2, C^*_q)+d(u_{s-1}, C^*_q)=|C_q|-3$,
then by $|C_q|-3>|C^*_q|/2$, there exists at least an edge
$v_xv_y\in E(C^*_q)$ such that $u_2v_x,u_{s-1}v_y\in E(G)$. By Claim
5, we can know that $d(x^*, C^*_q)+d(y^*, C^*_q)=|C_q|$. It follows
from the choices of $x^*,y^*$ in $C^*_p$ that $G_2+v_x+v_y-u_1-u_2$
contains a subdivision of $C_k$, $C^*_p+u_1+u_2-x^*-y^*$ contains a
subdivision of $C_p$ and $C^*_q+x^*+y^*-v_x-v_y$ contains a
subdivision of $C_q$. If $d(u_2, C^*_q)+d(u_{s-1}, C^*_q)=|C_q|-2$,
then by Claim 5 we have $d(x^*, C^*_q)+d(y^*, C^*_q)\geq |C_q|-1$.
Since $|C_q|-2-|C^*_q|/2\geq 2$, we can get that there exists at
least an edge $v_xv^+_x\in E(C^*_q)$ such that $u_2v_x,u_{s-1}v^+_x,
v^-_xx^*, v^{++}_xy^*\in E(G)$. Further, we have $d(x^*,
C^*_q-v_x-v^+_x)+d(y^*, C^*_q-v_x-v^+_x)\geq |C^*_q-v_x-v^+_x|-1$.
Thus we can obtain that $G_2+v_x+v^+_x-u_1-u_2$ contains a
subdivision of $C_k$, $C^*_p+u_1+u_2-x^*-y^*$ contains a subdivision
of $C_p$ and $C^*_q+x^*+y^*-v_x-v^+_x$ contains a subdivision of
$C_q$.

Next, we assume that $|C_q|=6$. Let $C^*_q=v_1v_2v_3v_4v_5v_6v_1$
with $v_1\in X$.

Suppose that $d(u_2, C^*_q)+d(u_{s-1}, C^*_q)\geq |C_q|-1$ and
$d(u_1, C^*_q)=|C^*_q|/2=3$. If there exists a vertex $v_x$ in
$V(C^*_q)\cap X$ such that $u_2v_x,u_sv_x\in E(G)$, then we can
obtain that $G_2+v_x-u_1$ contains a subdivision of $C_k$ and
$C^*_q+u_1-v_x$ contains a subdivision of $C_q$. Note that $d(u_s,
C^*_q)=1$. So we may assume that $d(u_2, C^*_q)=|C_q|/2-1$ and
$N(u_2, C^*_q)\cap N(u_s,C^*_q)=\emptyset$. Without loss of
generality, we let $u_sv_1, u_2v_3, u_2v_5\in E(G)$. We claim  that
$v_1v_4\in E(G)$. If not, then $P+v_1-u_1$ is a path of order $|P|$
and $C^*_q+u_1-v_1$ contains a subdivision of $C_q$ with
$|C^*_q+u_1-v_1|=|C_q|$ and
$|E(G[V(C^*_q+u_1-v_1)])|>|E(G[V(C^*_q)])|$, which contradicts the
choice of $C^*_1, C^*_2, \ldots, C^*_{k-1}$ in (iii). By Claim 5 we
have $d(x^*, C^*_q)+d(y^*, C^*_q)\geq |C_q|-2=4$. If $y^*$ has at
least two neighbors in $C^*_q$, then by $d(u_1, C^*_q)=|C^*_q|/2,
d(u_{s-1}, C^*_q)=|C^*_q|/2$ and the choice of $y^*$ in $C^*_p$,
there exists a vertex $v_y$ in $C^*_q$ such that $G_2+v_y-u_s$
contains a subdivision of $C_k$, $C^*_p+u_s-y^*$ contains a
subdivision of $C_p$ and $C^*_q+y^*-v_y$ contains a subdivision of
$C_q$. So we have $d(x^*,C^*_q)=3$ and $d(y^*,C^*_q)=1$. We claim
that $x^*y^*\in E(G)$. If not, then $P+v_1-u_1$ is a path of order
$|P|$, $C^*_q+x^*-v_1$ contains a subdivision of $C_q$ with
$|C^*_q+x^*-v_1|=|C_q|$ and
$|E(G[V(C^*_q+x^*-v_1)])|=|E(G[V(C^*_q)])|$, and $C^*_p+u_1-x^*$
contains a subdivision of $C_p$ with $|C^*_p+u_1-x^*|=|C_p|$ and
$|E(G[V(C^*_p+u_1-x^*)])|>|E(G[V(C^*_p)])|$, which contradicts the
choice of $C^*_1, C^*_2, \ldots, C^*_{k-1}$ in (iii). If $v_1y^*\in
E(G)$, then we can get that $G_2+v_2+v_3-u_1-u_s$ contains a
subdivision of $C_k$, $C^*_p+u_1+u_s-x^*-y^*$ contains a subdivision
of $C_p$ and $C^*_q+x^*+y^*-v_2-v_3$ contains a subdivision of
$C_q$. The cases $v_3y^*\in E(G)$ and $v_5y^*\in E(G)$ can be proved
similarly.

Suppose that $d(u_2, C^*_q)+d(u_{s-1}, C^*_q)\geq |C_q|-1$ and
$d(u_1, C^*_q)=d(u_s, C^*_q)=|C^*_q|/2-1=2$. By the symmetry, we may
assume that $d(u_2, C^*_q)=|C_q|/2=3$ and $d(u_{s-1}, C^*_q)\geq
|C^*_q|/2-1=2$. Without loss of generality, we let
$u_1v_2,u_1v_6,u_{s-1}v_2\in E(G)$. If $u_sv_1\in E(G)$, then it is
easy to see that $G_2+v_1-u_1$ contains a subdivision of $C_k$ and
$C^*_q+u_1-v_1$ contains a subdivision of $C_q$. Thus, we consider
$u_sv_1\not\in E(G)$ and by $d(u_s, C^*_q)=|C^*_q|/2-1=2$, we have
$u_sv_3,u_sv_5\in E(G)$. We claim that $v_1v_4\in E(G)$. If not,
then $P+v_3+v_4-u_1-u_2$ is a path of order $|P|$ and
$C^*_q+u_1+u_2-v_3-v_4$ contains a subdivision of $C_q$ with
$|C^*_q+u_1+u_2-v_3-v_4|=|C_q|$ and
$|E(G[V(C^*_q+u_1+u_2-v_3-v_4)])|>|E(G[V(C^*_q)])|$, which
contradicts the choice of $C^*_1, C^*_2, \ldots, C^*_{k-1}$ in
(iii). It follows from Claim 5 that $d(x^*, C^*_q)+d(y^*, C^*_q)\geq
|C_q|-3=3$. If $v_1y^*$ or $v_5y^*\in E(G)$, then
$G_2+v_2+v_3-u_1-u_s$ contains a subdivision of $C_k$,
$C^*_p+u_s-y^*$ contains a subdivision of $C_p$ and
$C^*_q+y^*+u_1-v_2-v_3$ contains a subdivision of $C_q$. If
$v_3y^*\in E(G)$, then $G_2+v_1+v_2-u_1-u_s$ contains a subdivision
of $C_k$, $C^*_p+u_s-y^*$ contains a subdivision of $C_p$ and
$C^*_q+y^*+u_1-v_1-v_2$ contains a subdivision of $C_q$. So we let
$d(y^*,C^*_q)=0$ and $d(x^*,C^*_q)=3$. Then we can see that
$G_2+v_1+v_2-u_1-u_s$ contains a subdivision of $C_k$,
$C^*_p+u_1-x^*$ contains a subdivision of $C_p$ and
$C^*_q+x^*+u_s-v_1-v_2$ contains a subdivision of $C_q$.

Suppose that $|C_q|-3\leq d(u_2, C^*_q)+d(u_{s-1},
C^*_q)\leq|C_q|-2$. From Claim 5, we can obtain that $d(x^*,
C^*_q)+d(y^*, C^*_q)\geq|C_q|-1$. First, we assume that $d(x^*,
C^*_q)+d(y^*, C^*_q)=|C_q|$. If there exists an edge $v_xv_y$ in
$C^*_q$ with $u_2v_x,u_{s-1}v_y\in E(G)$, then $G_2+v_x+v_y-u_1-u_s$
contains a subdivision of $C_k$, $C^*_p+u_1+u_s-x^*-y^*$ contains a
subdivision of $C_p$ and $C^*_q+x^*+y^*-v_x-v_y$ contains a
subdivision of $C_q$. So, we consider $d(u_2, C^*_q)=3$ or
$d(u_{s-1}, C^*_q)=3$. By the symmetry, we let $d(u_2, C^*_q)=3$.
Note that $d(u_s,C^*_q)>0$. This implies that there exists a vertex
$v_x\in V(C^*_q)$ such that $u_2v_x,u_sv_x\in E(G)$. Then
$G_2+v_x-u_1$ contains a subdivision of $C_k$, $C^*_p+u_1-x^*$
contains a subdivision of $C_p$ and $C^*_q+x^*-v_x$ contains a
subdivision of $C_q$. Next, we assume that $d(x^*, C^*_q)+d(y^*,
C^*_q)=|C_q|-1=5$. It follows from Claim 5 that $d(u_2,
C^*_q)+d(u_{s-1}, C^*_q)=|C_q|-2=4$. Since $d(x^*, C^*_q)+d(y^*,
C^*_q)=|C_q|-1=5$, we have $d(x^*, C^*_q)=3$ or $d(y^*, C^*_q)=3$.
If $d(x^*, C^*_q)=3$, then $d(y^*, C^*_q)=2$. We claim that
$x^*y^*\in E(G)$. If not, then by $d(u_s,C^*_q)\geq 1$, say
$u_sv_1\in E(G)$, we can  get that $P+v_1-u_1$ is a path of order
$|P|$, $C^*_q+x^*-v_1$ contains a subdivision of $C_q$ with
$|C^*_q+x^*-v_1|=|C_q|$ and $|E(G[V(C^*_q+x^*-v_1)])|\geq
|E(G[V(C^*_q)])|$, and $C^*_p+u_1-x^*$ contains a subdivision of
$C_p$ with $|C^*_p+u_1-x^*|=|C_p|$ and
$|E(G[V(C^*_p+u_1-x^*)])|>|E(G[V(C^*_p)])|$, which contradicts the
choice of $C^*_1, C^*_2, \ldots, C^*_{k-1}$ in (iii). From $d(u_2,
C^*_q)+d(u_{s-1}, C^*_q)=|C_q|-2=4$, we can deduce that there exists
an edge $v_xv_y\in E(C^*_q)$ such that $u_2v_x,u_{s-1}v_y\in E(G)$
and $G_2+v_x+v_y-u_1-u_s$ contains a subdivision of $C_k$,
$C^*_p+u_1+u_s-x^*-y^*$ contains a subdivision of $C_p$ and
$C^*_q+x^*+y^*-v_x-v_y$ contains a subdivision of $C_q$. Now we let
$d(y^*, C^*_q)=3$. If $d(u_1,C^*_q)\geq 2$ or
$N(u_{s-1},C^*_q)\setminus N(u_1,C^*_q)\neq \emptyset$, then we can
obtain that there exists an edge $v_xv_y\in E(C^*_q)$ such that
$u_2v_x,u_{s-1}v_y\in E(G)$ and $G_2+v_x+v_y-u_1-u_s$ contains a
subdivision of $C_k$, $C^*_p+u_s-y^*$ contains a subdivision of
$C_p$ and $C^*_q+u_1+y^*-v_x-v_y$ contains a subdivision of $C_q$.
If $d(u_1,C^*_q)=d(u_{s-1},C^*_q)=1$ and there exists the vertex
$v_y\in V(C^*_q)$ with $u_1v_y,u_{s-1}v_y\in E(G)$, then  by
$d(u_s,C^*_q)=3$ we can get that $G_2+v_y-u_s$ contains a
subdivision of $C_k$ and $C^*_q+u_s-v_y$ contains a subdivision of
$C_q$.

The proof of Theorem 3 is complete. \hfill$\Box$

%

\end{document}